\documentclass[11pt]{article}
\usepackage{amsmath, amsthm, amssymb}
\allowdisplaybreaks
\usepackage{bbm}
\usepackage{breqn}
\usepackage{amscd,amssymb}
\usepackage[all]{xy}
\usepackage{color}
\usepackage{comment}
\usepackage{appendix}
\usepackage{verbatim}
\usepackage{mathrsfs}
\usepackage{pdfsync}
\usepackage{graphicx,epsfig}

\usepackage{tikz}
\usetikzlibrary{matrix}
\usepackage{hyperref}
\usepackage{titlesec}
\usetikzlibrary{shapes,arrows}
\usepackage{caption}
\usepackage{subcaption}
\usetikzlibrary{positioning}
\usepackage{indentfirst}
\usepackage[T1]{fontenc}
\usepackage[ansinew]{inputenc}
\usepackage{float}
\usetikzlibrary{shapes,arrows}
\usepackage{setspace}
\usepackage{lipsum}

\date{}

\let\oldproofname=\proofname
\renewcommand{\proofname}{\rm\bf{\oldproofname}}

\newtheorem{thm}{Theorem}[section]
\newtheorem{cor}[thm]{Corollary}
\newtheorem{lem}[thm]{Lemma}
\newtheorem{prop}[thm]{Proposition}

\theoremstyle{definition}

\newtheorem{rmk}[thm]{Remark}
\newtheorem{defn}[thm]{Definition}

\newtheorem{conj}{Conjecture}[section]

\def\1{\mathbf{1}}

\newcommand{\N}{\mathcal{N}}

\newcommand{\frG}{\mathfrak{G}}
\newcommand{\frT}{\mathfrak{T}}

\begin{document}
\title{Neighborhood complexes, homotopy test graphs and a contribution to a conjecture of Hedetniemi}
\author{ Samir Shukla\footnote{Department of Mathematics, Indian Institute of Technology Bombay, India. samirshukla43@gmail.com} \footnote{The author was financially supported by Indian Institute of Technology Bombay, India. }}
\maketitle
\begin{abstract}

	The neighborhood complex $\N(G)$ of a graph $G$ were introduced by L. Lov{\'a}sz in his proof of Kneser conjecture. He  proved that for any graph $G$, 
	\begin{align} \label{abstract}
\chi(G) \geq conn(\N(G))+3.
	\end{align}
	
	In this article we show that for a class of exponential graphs the bound given in (\ref{abstract}) is tight. Further, we  show that the neighborhood complexes of these exponential graphs are spheres up to homotopy. 
	We were also  able to find a  class of exponential graphs, which are homotopy test graphs.

Hedetniemi's conjecture states  that the chromatic number of the categorical product of two graphs is the minimum of the chromatic number of the factors.  Let $M(G)$ denotes the Mycielskian of a graph $G$. We show that, for any graph $G$ containing $M(M(K_n))$ as a subgraph and  for any graph $H$, if  $\chi(G \times H) = n+1$, then $\min\{\chi(G), \chi(H)\} = n+1$.  Therefore, we enrich the family of graphs satisfying the Hedetniemi's conjecture.

\end{abstract}

\noindent {\bf Keywords} : Neighborhood complexes, exponential graphs, chromatic number.

\noindent 2010 {\it Mathematics Subject Classification:} primary 05C15, secondary 57M15

\vspace{.05in}

\hrule

\section{Introduction} \label{sec:introduction}

\subsection{Neighborhood complex} \label{subsec:neighborhood}
In 1978,  L. Lov{\'a}sz (\cite{l}) proved the famous Kneser conjecture (\cite{MK})
by using Borsuk-Ulam theorem. In his proof he introduced the concept of a simplicial complex $\N(G)$ called neighborhood complex for a graph $G$ and then related the topological connectivity of $\N(G)$ to the chromatic number of $G$.

\begin{thm}  (Lov{\' a}sz) \label{lovaszresult} For any graph $G$,
\begin{align} \label{lovaszbound}
\chi(G) \geq conn(\N(G)) +3.
\end{align}
\end{thm}

Here for a space $X$, $conn(X)$ is the largest integer $n$ such that $X$ is $n$-connected.
Lov{\'a}sz proved that in the case of Kneser graph the bound given in (\ref{lovaszbound}) is tight.  In \cite{BL}, Bj\"{o}rner and Longueville showed that the neighborhood complexes of a family of vertex critical subgraphs of
Kneser graphs -- the stable Kneser graphs, are spheres up to homotopy.  In \cite{BK}, Babson and Kozlov showed that neighborhood complex of complete graph $K_m$ is homotopy equivalent to the $(m-2)$-sphere $\mathbb{S^}^{m-2}$.  In \cite{NS}, Nilakantan and the author have studied the neighborhood complexes of the exponential graphs $K_{m}^{K_n}$. 
We have shown that $\N(K_m^{K_n})$ is homotopy equivalent to $\mathbb{S}^{m-2}$ for $m < n$. For all the above mentioned classes of graphs the bound given in (\ref{lovaszbound}) is sharp. It is a natural question to ask whether we can classify the graphs for which the inequality (\ref{lovaszbound}) is sharp. 
In this article we show that the bound given in (\ref{lovaszbound}) is sharp for a class of exponential graphs. Further, we show that the neighborhood complexes of these graphs are spheres up to homotopy.

	For $r \geq 1$, let $L_r$ denotes the path of length $r$, i.e., it is a graph with  vertex set $V(L_r) = \{0, \ldots, r\}$ and  edge set $E(L_r) = \{(i,i+1) \ |  \ 0 \leq i \leq r-1\}$. For  a subset $A \subseteq \{0, \ldots, r\}$,  $L_r(A)$ is the graph obtained from $L_r$ by adding a loop at $x$ for each $x \in A$. For a graph  $G$, we define the graph $G_A^r = (G \times L_r(A)) / \sim_{r}$, 	where $\sim_r$ is the equivalence which identifies all vertices whose second coordinate is $r$.  
For a graph G and $r \geq 2$, the graph $G_{\{0\}}^r = (G \times L_r(\{0\})) /  \sim_r$ is called the  {\it $r$-$th$  generalised Mycielskian} \footnote{The generalized Mycielskians (also known as cones over graphs) of a graph  were  introduced by Tardif (\cite{Tardif}), which are the natural generalization of Mycielskian.} of $G$ and we denote it by $M_r(G)$. The graph $M_2(G)$ is called the {\it Mycielskian} \footnote{The Mycielskian of a graph was introduced by Mycielski (\cite{Mycielski}) in 1955.} of $G$. In this article, if no confusion arises we denote $M_2(G)$ simply by  $M(G)$.

\textit{We say that a simple  graph $G$ satisfies the property $\mathcal{P}$, if the following is true: for any two disjoint edges $ e_1 = (v_1, w_1)$ and $e_2 = (v_2, w_2)$ of  $G$ if  $(v_2 , v_1)$ is an edge in $G$  then either $(w_2 ,v_1)$ or $ (w_2, w_1)$ is an edge  in $G$}.  

Here, by disjointness of edges $e_1 = (v_1, w_1)$ and $e_2 = (v_2, w_2)$, we mean  that $\{v_1, w_1\} \cap \{v_2, w_2\} = \emptyset$. Clearly, all the complete graphs satisfies the property $\mathcal{P}$. We prove the following.

\begin{thm} \label{main2}   Let  $2 \leq m < \chi(T)$ be a positive integer and let $T$  be  a connected graph satisfying the property $\mathcal{P}$. Then, for $K_m^T$ the bound given in (\ref{lovaszbound}) is sharp. Moreover, $\N(K_m^T) \simeq \mathbb{S}^{m-2}$. 
\end{thm}

\begin{thm} \label{cormain}
	Let $ n \geq 2, r \geq 1,  0 \leq i \leq r-1$ and  $A = \{0,\ldots, i\}$.   Let $T$  be a connected  graph satisfying  the property $\mathcal{P}$. Then for any $m \leq \chi(T)$,
	$\N(K_m^{T_A^r}) \simeq \mathbb{S}^{m-2}$. In particular, $\N(K_m^{M_r(K_n)}) \simeq \mathbb{S}^{m-2}$ for all $m \leq n$.
	
\end{thm}

The following  corollary is a direct application of  Theorem \ref{cormain}.

\begin{cor} \label{corsharpe}
	Let $T$  be  a  connected graph satisfying the property $\mathcal{P}$. Let $2 \leq m \leq \chi(T), r \geq 1$, $0 \leq i \leq r-1$ and $A = \{0, \ldots, i\}$. Then, for  $K_{m}^{T_A^r}$ the  bound given in  (\ref{lovaszbound}) is sharp. 
	
\end{cor}

\begin{thm}\label{doubesharp}
	Let $n \geq 2$ and $m \leq n$ be positive integers. Then,  inequality (\ref{lovaszbound}) is sharp for $K_m^{M(M(K_n))}$. Moreover,  $\N(K_m^{M(M(K_n))}) \simeq \mathbb{S}^{m-2}$. 
	\end{thm}
\subsection{Homotopy test graph} \label{subsec:testgraph}
	
The Hom complex  were  originally defined by Lov{\'a}sz  to give a topological obstruction to graph coloring and were first mainly investigated by Babson and Kozlov in \cite{BK}. The Hom complex Hom$(G, H)$ for  graphs $G$ and $H$ is a polyhedral complex, whose  $0$-dimensional cells are the graph  homomorphisms from $G$ to $H$. It is known that  Hom$(K_2, G)$ and $\N(G)$ are homotopy equivalent (\cite{BK}). 
	
\begin{defn}\cite{dk}
	A graph $T$ is called a {\it homotopy test graph} if the inequality
		\begin{align} \label{lovaszinequality}
		\chi(G) \geq \chi(T) + conn(\text{Hom(T,G)}) +1,
		\end{align}
		holds for all graph $G$.
		\end{defn}
	%The conjecture \ref{lovaszconjecture} is not true for all pair of graphs $T$ and $G$. In ... Hoory and Linial gave a counter example to this conjecture. 
	
	Since Hom$(K_2, G) \simeq \N(G)$, it follows  from Theorem \ref{lovaszresult} that $ K_2$ is a homotopy test graph. In \cite{BK}, Babson and Kozlov generalize this result and showed that  all complete graphs are homotopy test graphs. 
	Some other examples of	homotopy test graphs are given by  odd cycles $C_{2r+1}$ \cite{BK1}, bipartite
graphs \cite{TM} and some of the stable Kneser graphs \cite{schultz}. For further development and related topics,
we refer to \cite{antonandschultz, Hoory,Kozlov, dk}.

In this article, we  prove the following.

\begin{thm}\label{doublenew}
Let $n \geq 2$  and let $G$ be a graph. If  $M(M(K_n))$ is a subgraph of $G$, then 
 $\chi(K_m^G) = m$ for $m \leq n+1$.
\end{thm}

The following corollary is an application of Theorem \ref{doublenew}.
\begin{cor} \label{testgraph}
	Let $n \geq 2$ and  let  $G$ be a graph. 
	 If $M(M(K_n))$ is a subgraph of $G$, then $K_m^{G}$ is a homotopy test graph for $m \leq n+1$.
	\end{cor}

	 	\subsection{Hedetniemi's conjecture} \label{hedetnimi}
	One of the famous open problem in graph theory is the following conjecture of Hedetniemi about the chromatic number of categorical product of graphs (for categorical product of graphs, see Definition \ref{def:product}).
	\begin{conj}\cite{HD} \label{Hedetniemi}
		For any two graphs $G$ and $H$, 
			$\chi(G \times H) = \text{min}\{\chi(G), \chi(H)\}$.
		
			\end{conj}
		
		Since we have projections $G \times H \to G$ and $G \times H \to H$, it is straightforward to verify that 
		$\chi(G \times H ) \leq \text{min}\{\chi(G), \chi(H)\}$. The difficulty lies in deriving the other inequality i.e., to  prove that  $\chi(G \times H ) \geq \text{min}\{\chi(G), \chi(H)\}$. It is easy to prove that Conjecture \ref{Hedetniemi} is true for graphs of chromatic numbers $2$ and $3$.
In \cite{ZS}, El-Zahar and Sauer, showed that the chromatic number of the product of two $4$-chromatic graph is $4$. The following are the some of the  results in the support of the Conjecture \ref{Hedetniemi}. 
	
	\begin{thm}\cite{BEL} \label{thm:BEL}
		Let $G$ be graph such that every vertex of $G$ is in an $n$-clique. For every graph $H$, if $\chi(G \times H) = n$, then $\min \{\chi(G), \chi(H)\} =  n$. 
		
		\end{thm}
	
	\begin{thm} \cite{DSW}\cite{EW}
		Let $G$ and $H$ be connected graphs containing $n$-cliques. If $\chi(G \times H) = n$, then $\min\{\chi(G), \chi(H)\} = n$.
	\end{thm}

	\begin{thm}\cite[Theorem $3$]{moroli} \label{thm:moroli}
		Let $G$  be be graph such that the clique number of $G$ is $n \geq 2$.
Let $A$ be the set of vertices of $G$ that are not contained in an $n$-clique. Let there exists a 
coloring  with $n+1$ colors $1, \ldots, n+1$ of $G$  such that each vertex of $A$ is colored by either $n$ or $n+1$.  For every graph $H$, if $\chi(G \times H) = n$, then $\min \{\chi(G), \chi(H)\} =  n$. 
		\end{thm}
	In all above results, for a  pair of graph $G$ and $H$, the Hedetniemi's conjecture is true, when $\chi(G) = n+1$ and $G$ contains an $n$-clique. It is natural to ask, what can we say, if  $\chi(G) = n+2$ and the size of the maximal clique in $G$  is  $n$. 	In this article we prove the following, which is an immediate consequence of Theorem \ref{doublenew}.
	
%	Let $M_r(G)$ denotes the $r$-th genaralized Mycielskian of $G$ (for definition of $M_r(G)$ see section \ref{sec:results}). First we show that, if 
% $M_r(K_n)$  is a subgarph of a graph $G$, then for any $H$, $\chi(G \times H) = n  \implies \min\{\chi(G), \chi(H)\} = n$. 
% Secondly, we show that, if  $M_2(M_2(K_n))$ is a subgraph of a graph $G$, then for any $H$, $\chi(G \times H) =  n+1 \implies \min\{\chi(G), \chi(H)\} = \chi(T)+1 $ (see Corollary \ref{Hedet}).
%	
	
	\begin{thm}\label{Hedet}
		Let $n \geq 2$ be positive integer. Let  $G$  be  a graph containing $M(M(K_n))$ as a subgraph. For every graph $H$, if $\chi(G \times H) = n+1$, then $\min \{\chi(G), \chi(H)\} =  n+1$. 
	\end{thm}

 Some other special
cases of the Conjecture \ref{Hedetniemi} have also been proved to be true \cite{meysam, moroli, BEL,  DSW, SZ,  TZ, DT,  EW}.
For further development and related research, we refer to the survey articles \cite{sauer, CT, XZ}.

\section{Preliminaries} \label{sec:prel}
%\begin{case}
%	\end{case}
%	\begin{subcase}
%		\end{subcase}
%	
%
%\begin{subcase}
%	\end{subcase}
%
%\begin{case}
%	\end{case}
%
%\begin{subcase}
%	\end{subcase}

A  graph $G$ is a  pair $(V(G), E(G))$,  where $V(G)$  is the set of vertices of $G$  and $E(G) \subseteq V(G)\times V(G)$ denotes the set of edges.
If $(x, y) \in E(G)$, it is also denoted by $x \sim y$ and $x$ is said to be adjacent to $y$.
All the graphs in this article are finite, undirected, i.e., $(x,y) \in E(G)$ implies $(y, x) \in E(G)$ and without multiple edges.

A graph $G$ is said to be {\it simple} if $(x,y) \in E(G)$ implies that $x \neq y$. A {\it subgraph} $H$ of $G$ is a graph with $V(H) \subseteq V(G)$ and $E(H) \subseteq E(G)$.
For a subset $U \subseteq V(G)$, the induced subgraph $G[U]$ is the subgraph whose set of vertices  $V(G[U]) = U$
and the set of edges
$E(G[U]) = \{(a, b) \in E(G) \ | \ a, b \in U\}$.

For a positive integer $r$, {\it cyclic graph} $C_r$ is the graph with $V(C_r) = \{1, \ldots, r\}$ and $E(C_r) = \{(i, i+1) \ | \ 1 \leq i \leq r-1\} \cup \{(1,r)\}$.

The {\it
	neighborhood  of $v \in V(G)$} is defined as $N_G(v)=\{ w \in V(G) \ |  \
(v,w) \in E(G)\}$.  If $A\subseteq V(G)$, the neighborhood of $A$
is defined as $N_G(A)= \{x \in  V(G) \ | \ (x,a) \in E(G)
\ \forall \ a \in A \}$.  %The {\it degree} of a vertex $v$ is $|N(v)|$ and it is denoted by $d(v)$.
 
  A {\it graph homomorphism} from $G$ to $H$ is a set map $\phi: V(G) \to V(H)$ which preserve the edges, i.e, $(x,y) \in E(G) \implies (\phi(x), \phi(y)) \in E(H)$. A graph homomorphism $f$ is called an {\it isomorphism} if $f$ is bijective and $f^{-1}$ is also a graph homomorphism. Two graphs are called {\it isomorphic},  if there exists an isomorphism between them. If $G$ and $H$ are isomorphic, we write $G \cong H$.
  
  The {\it complete graph} $K_n$ is a graph with 
$V(K_n) = \{1, \ldots, n\}$  and $E(K_n) =\{(i,j) | i \neq j\}$. A graph isomorphic to  $K_n$ is called an {\it $n$-clique} or a {\it clique of size $n$}. The {\it chromatic number} of a simple graph is defined by 
$\chi(G): = \min\{n \ | \ \exists \ \text{a graph homomorphism from} \ G \ \text{to} \ K_n  \}$.

The
{\it clique number} $\omega(G)$ of a graph $G$ is the largest integer $k$ such that $G$ contains an $k$-clique as a subgraph. 

\begin{defn}[Categorical product] \label{def:product}
		The {\it categorical product} of two graphs $G$ and $H$, denoted by $G\times H$ is the graph where $V(G\times H)=V(G)\times V(H)$ and  $(g,h) \sim (g',h')$ in $G\times H$  if $g \sim g'$  and $h \sim h'$ in $G$ and $H$ respectively.
	\end{defn}
	
	\begin{defn}[Exponential graph] Let $G$ and $H$ be graphs. The {\it exponential graph} $\displaystyle H^{G}$ is a graph where $\displaystyle V(H^{G})$ contains all the set maps from $V(G)$ to $V(H)$ and any two vertices  $f$  and $f'$ in $\displaystyle V(H^{G})$ are  said to be adjacent, if $ v \sim v'$  in $G$ implies that $f(v) \sim f'(v')$ in $H$.
	
	\end{defn}

More details about product of graphs and exponential graphs can be  found in 	\cite{GR,HN}.
	
	A topological space $X$ is said to be $k$-{\it connected} if every continuous map from a $m$-dimensional sphere $\mathbb{S}^m \to X$ can be extended to a continuous map from the $(m+1)$-dimensional ball $\mathbb{B}^{m+1} \to X$ for $m = 0, 1, \ldots , k.$ The {\it connectivity} of $X$ is the largest $k$ such that $X$ is $k$-connected and it is denoted by $conn(X)$. If $X$ is a non empty
	disconnected space, it is said to be $-1$ connected.

	A {\it finite abstract simplicial complex X} is a collection of
	finite sets where $\tau \in X$ and $\sigma \subset \tau$,
	implies $\sigma \in X$. The elements  of $X$ are called the  {\it simplices}
	of $X$. If $\sigma \in X$ and $|\sigma |=k+1$, then $\sigma$ is said
	to be $k\,dimensional$.

	\begin{defn}
	The {\it neighborhood complex} $\N(G)$ of a graph $G$ is the abstract simplicial complex whose simplices are those subsets of 
	vertices of $G$, which have a common neighbor, i.e., $\N(G) = \{A \subseteq V(G) \ | \ N_G(A) \neq \emptyset\}$.
	\end{defn}
	
	In this article we consider $\N(G)$ as a topological space, namely its  geometric realization. For  definition of geometric realization and details about simplicial complexes, we refer to  book \cite{dk} by  Kozlov.
	
		\begin{defn}\cite[Definition $1.2$]{BK}
	For any two graphs $G$ and $H$, Hom complex  Hom$(G,H)$ is the polyhedral
	complex whose cells are indexed by all functions $\eta: V(G) \to
	2^{V(H)}\setminus \{\emptyset\}$, such that if $(v,w) \in E(G)$ then
	$\eta(v) \times \eta(w) \subseteq E(H)$.

	\end{defn}

	 Let $G$ be a graph and $N_G(u) \subseteq N_G(v)$  for two distinct vertices $u$ and $v$ of $G$. The graph $G\setminus \{u\}$ is called a {\it fold} of $G$.  Here,  $V(G\setminus
	\{u\}) = V(G) \setminus \{u\}$ and the edges in the subgraph
	$G\setminus \{u\}$ are all those edges of $G$ which do not contain
	$u$. In this case,  the map $V(G) \to V(G) \setminus \{u\}$, which sends $u$ to $v$ and fixes all other vertices, is a graph homomorphism. Thus, we conclude that  $\chi(G \setminus\{u\}) = \chi(G)$.

From \cite[Theorem $3.3$]{kozlov} we have the following result which allows us to replace a graph by a subgraph in the $\text{Hom}$ complex.

\begin{prop}\label{prop:folding}
	Let $G \setminus \{v\}$ be a fold of $G$ and let H be some graph. Then
	Hom$(G,H) \simeq$ Hom$(G \setminus\{v\} ,H)$  and  Hom$(H,G) \simeq $
	Hom$(H,G \setminus \{v\} )$. 
	\end{prop}

From \cite[Proposition 3.5]{ad} we have a relationship between the exponential graph and the categorical product in the $\text{Hom}$ complex.

\begin{prop}\label{prop:shifting}
	Let $G$, $H$ and $K$ be graphs.  Then
	$$
	\displaystyle  \text{Hom}(G \times H, K)  \simeq \text{Hom}(G, K^{H}).
	$$

\end{prop}

\section{Proofs }

%\begin{lem}
%	Let $A$ and $B$ be  graphs. Let $a \in  V(G)$  and $b \in V(H)$.  If $f \in V(B^A)$ such that $f(x) \neq b$ for all $x \in N_A(a)$. Then $N_{B^A} (f) \subseteq N_{B^A}(\tilde{f})$, where $\tilde{f}: V(A) \to V(B)$ defined by 
%	
%	\begin{center}
%		$\tilde{f}(x)= \begin{cases}
%		b, & \text{if}  \ x  = a ,\\
%		f(x), & \text{otherwise}.\\
%		\end{cases}$
%	\end{center}
%	 
%	\end{lem}

Throughout this article, for any positive integer $n$, we denote the set $ \{1, \ldots, n\}$ by $[n]$. For a map $f$, we let Im $f$ denote the image set of $f$. 

A simple graph $G$ is called {\it perfect} if, for every induced subgraph $H$ of $G$, the
number of vertices in a maximum clique is $\chi(H)$. The  {\it compliment graph} $G^c$  of $G$ is the graph on same vertex set as of $G$ and $E(G^c) = \{ (x,y) \ | \ (x,y) \notin E(G) \}$. The well known Strong Perfect Graph Theorem (\cite{Strongperfect}) says that $G$ is perfect if and
only if no induced subgraph of $G$ or $G^c$ is an odd cycle of length greater
than three. Using this characterization of perfect graphs we show that any graph satisfying the property $\mathcal{P}$ (defined in section \ref{subsec:neighborhood}) is perfect.

\begin{prop}  \label{perfect} 
	Every graph satisfying the property $\mathcal{P}$ is  perfect.
\end{prop}

\begin{proof} Let $T$ be  a graph satisfying the  property $\mathcal{P}$. It is clear from the definition of property  $\mathcal{P}$ that no induced subgraph of $T$ is isomorphic to  an odd cycle of length greater than $3$.
	Suppose $T^c$ has an induced subgraph $H$ isomorphic to an odd cycle of length $r \geq 5$. Without loss of generality we assume that $V(H) = \{1, \ldots, r\}$ and  $E(H) = \{(i,i+1) | 1 \leq i \leq r-1\} \cup \{(1,r)\}$. Then, clearly $(1,3)$,$(2,5)$ and $(3,5)$ are edges in $T$, but neither $(2,1)$ nor $(2,3)$ is an edge in $T$, which is a contradiction to the fact that $T$ is satisfying the  property $\mathcal{P}$. 
\end{proof}

\subsection{Proofs of the results  of Section \ref{subsec:neighborhood}}

We first prove the following lemma, which plays a crucial role in the proof of Theorems \ref{main2} and  \ref{cormain}.

\begin{lem} \label{fold2} 
	
Let $m$ be a positive integer and 	let $T$ be a connected   graph having the property $\mathcal{P}$.  Let $f \in V(K_m^T)$ such that $f(v) = f(w) = a$ for some edge $(v,w) \in E(G)$. Then,  there exists $\tilde{f}$ such that $N_{K_m^T} (f) \subseteq N_{K_m^T}(\tilde{f})$ and $\tilde{f}(x) = a$ for all $x \in V(T)$. 
\end{lem}

\begin{proof}
Define  $f^1 : V(T) \to [m]$  by   
	\begin{center}
		$f^1(x)= \begin{cases}
		a, & \text{if}  \ x \in N_T(v) ,\\
		f(x), & \text{otherwise}.\\
		\end{cases}$
	\end{center}
		Let $g \in N_{K_m^T}(f)$. We first show that $g \in N_{K_m^T}(f^1)$. Since $f^1(x) = f(x)$ for all $x \notin N_T(v)$, observe that to prove $g \sim f^1$, it is enough to show that for each $x \in N_T(v)$ and $y \sim x$, $g(x) \sim f^1(y)$ and $g(y) \sim f^1(x)$ in $K_m$. Let $x \in N_T(v)$ and $y \sim x$ in $T$. If $y \in N_T(v)$, then $f^1(y) = a$. In this case, since $x \sim v$ in $T$ and  $f\sim g$, we see that $g(x) \sim f(v) = a = f^1(y)$.  Now, if $y \notin N_T(v)$, then $f^1(y) = f(y)$ and therefore $g(x) \sim f(y) = f^1(y)$.  Hence, $g(x) \sim f^1(y)$ in $K_m$. 
		
		We now show that $g(y) \sim f^1(x)$ in $K_m$. Since, $f^1(x) = a$, it is sufficient to show that $g(y) \neq a$. If $x = w$, then $g \sim f$ implies that  $ g(y) \sim f(w)$ (in $K_m$). Hence, $g(y) \neq f(w) = a$. So assume that $x \neq w$. If $y \in \{v, w\}$, then since $(v, w) \in E(T)$ and $ f(v) = f(w) = a$, we see that  $g \sim f \implies g(y) \neq a$.  If $y \notin \{v, w\}$, then since $(v,x)$, $(x, y)$ belong to $E(T)$ and $T$  is satisfying the property $\mathcal{P}$, we conclude that either $(y,v)$ or $(y,w)$ is an edge in $T$. Now, $f \sim g$ implies that $g(y) \neq a$.

	Thus, $g \sim f^1$ and therefore $N_{K_m^T}(f) \subseteq N_{K_m^T}(f^1)$.  Now, we apply the above argument  for $f^1$ and get a $f^2 \in V(K_m^T)$ such that $N_{K_m^T} (f^1) \subseteq N_{K_m^T}(f^2)$
	and $f^2(x) = a$ for all $x \in N_T(v) \bigcup\limits_{y \in N_T(v)}N_T(y)$.  Since $T$ is a finite connected graph, after finite number of steps,  we get a $f^k \in V( K_m^T)$ such that $N_{K_m^T}(f) \subseteq N_{K_m^T}(f^1)$, $N_{K_m^T}(f^1) \subseteq N_{K_m^T}(f^2), \ldots, N_{K_m^T}(f^{k-1}) \subseteq N_{K_m^T}(f^k)$ and $f^k(x) = a$ for all $x \in V(T)$. We take $\tilde{f} = f^k$.

%From Claim \ref{claim}, $N_{K_m^G}(f) \subset N_{K_m^G}(f_1)$ and $f_1(x) = a$
%for all $x \in N_G(v) \cup \{v\}$. Now we can apply Claim \ref{claim} on $f_1$ and get a $f_2$ such that $N_{K_m^G} (f_1) \subset N_{K_m^G}(f_2)$
% and $f_2(x) = a$ for all $x \in N_G^2(v)$. Since $G$ is finite connected graph, using  Claim \ref{claim} we get $f_k \in K_m^G$ such that $N_{K_m^G}(f) \subset N_{K_m^G}(f_1)$, $N_{K_m^G}(f_1) \subset N_{K_m^G}(f_2), \ldots, N_{K_m^G}(f_{k-1}) \subset N_{K_m^G}(f_k)$ and $f_k(x) = a$ for all $x \in V(G)$. Hence, $G$ is folded on $G\setminus \{f, f_1, \ldots, f_{k-1}\}$. Repeating this process for all non constant maps $f$, $K_m^{G}$ is folded onto an induced subgraph $G'$, where  $V(G') = \{f \in K_m^G \ | \ f(v) = f(w) \  \forall \ v, w \in V(G)\}$

\end{proof}

\begin{proof}[Proof of Theorem \ref{main2}]
%	Fix $2 \leq m < \chi(T)$. Since, $m < \chi(T)$, there exists no graph homomorphism from $T$ to $K_m$. From Lemma \ref{fold2}, $K_m^{T}$ folded onto a subgraph $G'$ where $V(G') = \{f \in K_m^T \ | \ f(v) = f(w) \  \forall \ v, w \in V(T)\}$. Observe that $G'$ is isomorphic to the complete graph $K_m$. Hence $\N(K_m^G) \simeq \N(G') \simeq \N(K_m) \simeq \mathbb{S}^{m-2}$.
	
	Since, $m < \chi(T)$, there exists no graph homomorphism from $T$ to $K_{m}$. Hence, for any $f \in V(K_{m}^{T})$, there exists and edge $(v, w) \in E(T)$ such that $f(v) = f(w)$ and therefore from Lemma \ref{fold2}, $K_{m}^{T}$ is folded onto a subgraph $\frT_1$ where each vertex of $\frT_1$ is a constant map from $V(T)$ to $[m]$.  Observe that $\frT_1$ is isomorphic to the complete graph $K_{m}$. Since, folding preserve the chromatic number, $\chi(K_{m}^T) = \chi(\frT_1) = m$. 
	
	Since,  $\N(G) \simeq $ Hom$(K_2, G)$, from Proposition \ref{prop:folding}, we conclude that $\N(K_m^T) \simeq \N(\frT_1) \simeq \N(K_m) \simeq \mathbb{S}^{m-2}$. Now, since $conn(\mathbb{S}^{m-2}) = m-3$ result follows.
\end{proof}

\begin{rmk} In the proof of Theorem  \ref{main2}, we have shown that $\chi(K_m^T) = m$ for $m < \chi(T)$. But, this is directly follows from Theorem \ref{thm:BEL}, by using the following argument.  
	 From Proposition \ref{perfect}, $T$ is a perfect graph and therefore it contains a clique of order $\chi(T)$. Hence, Theorem \ref{thm:BEL} implies that $\chi(T \times H) = \min \{\chi(T), \chi(H)\}$ for any graph $H$. In particular $\chi(T \times K_m^T) = \min\{\chi(T), \chi(K_m^T)\}$. Since the map $\phi: T \times K_m^T \to K_m$ defined by $\phi(x, f) = f(x)$ is a graph homomorphism, we see that $\chi(T \times K_m^T) \leq  m$. But, since $\chi(K_m^T) \geq m$ and $\chi(T) > m$, $\chi(T \times K_m^T) = \min\{\chi(T), \chi(K_m^T)\}$ implies that $\chi(K_m^T) = m$.
	\end{rmk}

In the rest of the section, we assume that $T$ is a connected graph satisfying the property $\mathcal{P}$, $m$ and $r$ are  positive integers, $A \subseteq \{0, \ldots, r-1\}$ and  $\mathfrak{T} = T_A^r$. For any $f \in V(K_m^{\frT})$ and $0 \leq l \leq r-1$, we define $f_l : V(T) \to [m]$ by 
$f_l(x) = f((x,l))$ for all $x \in V(T)$.

\begin{lem} \label{foldgeneral}
Let $f \in V(K_m^{\frT})$ such that $f_l(v) = f_l(w) =a$ for some $(v,w) \in E(T)$ and some $0 \leq l \leq r-1$.  There exists  $\tilde{f}\in V(K_m^{\frT})$ such that $\tilde{f}_l(x) =  a$ for all $x \in V(T)$ and $N_{K_m^{\frT}}(f) \subseteq N_{K_m^{\frT}} (\tilde{f})$.
	\end{lem}

\begin{proof}
	The proof is similar to that of  proof of the Lemma \ref{fold2}. 
	Define $f^1 : V(\frT) \to [m]$ by  
	\begin{center}
		$f^1(x)= \begin{cases}
		a, & \text{if}  \ x \in N_T(v) \times \{l\} ,\\
		f(x), & \text{otherwise}.\\
		\end{cases}$
	\end{center}
	
	Let $h \in N_{K_m^{\frT}}(f)$. We first show that $h \sim f^1$ in $K_m^{\frT}$. Here, it is enough to show that for each $x \in N_{T}(v) \times \{l\}$ and $y \in N_{\frT}(x)$, $h(x) \sim f^1(y)$ and $h(y) \sim f^1(x)$ (in $K_m$). Let $x= (x_1,l)$ and $y = (y_1, j)$ for some $x_1, y_1 \in V(T)$, $1 \leq j  \leq r$ such that $x \in N_{T}(v) \times \{l\}$ and $y \in N_{\frT}(x)$.  Since, $(x,y) \in E(\frT)$, we see that $(j,l) \in E(L_r(A))$ and $(x_1, v) \in E(T)$. 
	If $y \in N_T(v) \times \{l\}$, then $f^1(y) = a$. In this case, $y \sim x$ in $\frT$ and $x_1 \sim v$ in $T$ implies that  $x \sim (v,l)$ in $\frT$. Hence,  $h\sim f$ implies that $h(x) \sim f((v,l)) = f_l(v) = a = f^1(y)$.  Now, if $y \notin N_T(v) \times \{l\}$, then $f^1(y) = f(y)$ and therefore $h(x) \sim f(y) = f^1(y)$.  Hence, $h(x) \sim f^1(y)$ in $K_m$.

	Here, $f^1(x) = a$. Hence, to show that $h(y) \sim f^1(x)$, it is enough to prove that $h(y) \neq a$. If $x_1 = w$, then $x = (w,l)$ and therefore $h \sim f$ implies that $h(y) \neq  f(x) = f((w,l) )= f_l(w)= a$.  If $y_1 \in \{v, w\}$, then either $((y_1,j), (v,l)) \in E(\frT)$ or $((y_1,j), (w,l)) \in E(\frT)$.
	Since, $f((v,l)) = f((w,l)) = a$ and  $h \sim f$, we see that  $h((y_1,j)) \neq a$.  So, assume that $x _1 \neq w$ and $y_1 \notin \{v, w\}$. Since $(w,v), (v, x_1)$, $(x_1, y_1)$ are edges in $T$  and $T$  is satisfying  the property $\mathcal{P}$, we conclude that either $(y_1, v)$ or $(y_1, w)$ is an edge in $T$. Hence, either $((y_1, j), (v,l)) \in E(\frT)$ or $((y_1, j), (w,l)) \in E(\frT)$. Now, $f \sim h$ implies that $h((y_1,j)) \neq a$.

	Thus $N_{K_m^{\frT}}(f) \subseteq N_{K_m^{\frT}}(f^1)$.  We now  apply the above argument  on $f^1$ and get a $f^2$ such that $N_{K_m^{\frT}} (f^1) \subseteq N_{K_m^{\frT}}(f^2)$
	and $f^2((x,l)) = a$ for all $x \in N_T(v) \bigcup\limits_{y \in N_T(v)}N_T(y)$. Since $T$ is finite connected graph, after finite number of steps,  we get a $f^k \in K_m^{\frT}$ such that $N_{K_m^{\frT}}(f) \subseteq N_{K_m^{\frT}}(f^1)$, $N_{K_m^{\frT}}(f^1) \subseteq N_{K_m^{\frT}}(f^2), \ldots, N_{K_m^{\frT}}(f^{k-1}) \subseteq N_{K_m^{\frT}}(f^k)$ and $f^k((x,l)) = a$ for all $x \in V(T)$. We take $\tilde{f} = f^k$.
	
\end{proof}

%From Lemma \ref{foldgeneral}, $K_m^{\mathfrak{G}}$ is folded onto an induced subgraph $\mathfrak{G}_1$ such that any $f \in V(\mathfrac{G}_1)$ and $1 \leq l \leq n$ either $f((x,l)) = f((y, l)) $ for all $x, y \in V(T)$ or $f(x,l) \neq f((y,l))$ for all $(x, y) \in E(T)$.

\begin{lem} \label{neighborhood}
	 Let $0 \leq l \leq r-1$ and let $f \in V (K_m^{\frT})$ such that the map $f_l$ is a graph homomorphism from $T$ to $K_m$. 
	 If $m \leq \chi(T)$, then $N_{K_m^{\frT}}(f) = \emptyset$. 

\end{lem}

\begin{proof}
	
		 Let $n = \chi(T)$. From Proposition \ref{perfect}, $T$ is a perfect graph and therefore it contains a subgraph isomorphic to the complete graph  $K_n$. Without loss of generality we can assume that $K_n$ is a subgraph of $T$, i.e., $V(K_n) = [n] \subseteq V(T)$.

	Suppose $N_{K_m^{\frT}}(f) \neq \emptyset$ and  let $h \in N_{K_m^{\frT}}(f)$.  For $0 \leq i \leq r-1$, let $\tilde{f}_i$ and $\tilde{h}_i$ are restrictions of $f_i$ and $h_i$ on $K_n$ respectively. Since $f_l$  is a graph homomorphism from $T$ to $K_m$, $\tilde{f}_l$ is also a graph homomorphism from $K_n$ to $K_m$ and therefore    $\tilde{f}_l(i) \neq \tilde{f}_l(j)$ for all $1 \leq i \neq j \leq n$. 
	
 Since $m \leq \chi(T) = n$ and $\tilde{f}_l$ is a graph homomorphism, we conclude that $m = n$ and 
	 Im $\tilde{f}_l =[m]$. For any $1 \leq i \neq j \leq n$, since $(i,l+1) \sim (j,l)$ in $\frT$ and $f \sim h$ in $K_m^{\frT}$, we have $h((i,l+1)) \neq f((j,l))$. Hence, $\tilde{h}_{l+1}(i) \neq \tilde{f}_l(j)$ for $i \neq j$ and therefore  $\tilde{h}_{l+1} (i)= \tilde{f}_l(i)$ for all $1 \leq i \leq n$. Thus, $\tilde{h}_{l+1} = \tilde{f}_l$ and Im $\tilde{h}_{l+1} =[m]$.

For any $i \neq j$, since $(i, l+2) \sim (j, l+1)$, we see that $f((i,l+2)) \neq h((j,l+1 ))$. Using the fact that $\tilde{h}_{l+1} = \tilde{f}_l$, we conclude that $\tilde{f}_{l+2} = \tilde{h}_{l+1} = \tilde{f}_l$ and therefore Im $\tilde{f}_{l+2}= [m]$.	By applying the similar argument, since $(i,l+2) \sim (j,l+3)$ for all $1 \leq i \neq j \leq n$, we conclude that 
Im $\tilde{h}_{l+3} = [m]$ and $\tilde{h}_{l+3} = \tilde{f}_{l+2} = \tilde{h}_{l+1} = \tilde{f}_l$.
	
By 	repeating the above argument, we observe that Im $\tilde{f}_{l+t} = [m]$ for all even $t$ such that $l+ t \leq r-1$ and Im $\tilde{h}_{l+s}= [m]$, for all  odd $s$ such that $l+s \leq r-1$. Hence, either  $\text{Im} \ \tilde{f}_{r-1}= [m]$ or $\text{Im} \ \tilde{h}_{r-1} = [m]$.
	
	Let $w = (\ast, r)$ be the vertex of $T_A^r$, obtained by identifying all vertices  whose second coordinate is $r$.
Since, $w \sim (i,r-1)$ for all $1 \leq i \leq n$ and either $\text{Im} \ \tilde{f}_{r-1}= [m]$ or Im $\tilde{h}_{r-1} = [m]$, we see that $f \sim h \implies$ either  $h(w) \notin [m]$ or $f(w) \notin [m]$, which is a contradiction. Thus $N_{K_m^{\frT}}(f) = \emptyset$.

	\end{proof}

Using the above two lemmas we prove the following more general result, from which  proof of Theorem \ref{cormain} follows easily.

\begin{thm} \label{generalmain}
Let $r$ be  a positive integer and 	let $T$  be  a connected graph having the property $\mathcal{P}$. Then for any   $A \subseteq \{0, 1\ldots, r-1\}$  and  $2 \leq m \leq \chi(T)$, 
	$$
	\N(K_m^{T_A^r}) \simeq \N(K_m^{L_r(A)}).
	$$

\end{thm}

 		\begin{proof}
 			
 		By 	using Lemma \ref{foldgeneral}, $K_m^{T_A^r}$ is folded onto an induced subgraph $\frT_1$
 		where $f \in V(\mathfrak{T}_1)$ if and only if, for each $0 \leq l \leq r-1$, the map $f_l : V(T) \to [m]$ defined by $f_l(x) = f((x,l))$ is either a constant map or a graph homomorphism from $T$ to $K_m$. For any $f \in V(\frT_1)$ such that $f_l$ is non-constant for some $0 \leq l \leq r-1$, $N_{\frT_1}(f) = \emptyset$ from Lemma \ref{neighborhood}. 
 		Hence,  $\N(\frT_1) \simeq \N(\frT_2)$, where $\frT_2$ is the induced subgraph of $\frT_1$ on the set of vertices $V(\frT_2) = \{f \in V(\frT_1) \ | \ f_l \ \text{is a constant map} \ \forall \ 0 \leq l \leq r-1\}$.
 
 	Since $\chi(T) \geq 2$, $T$ contains at least one edge and using this fact it is  easy to see that   $\frT_2 \cong K_m^{L_r(A)}$. Since $\N(G) \simeq $ Hom$(K_2, G)$, using Proposition \ref{prop:folding}, we conclude that $\N(K_m^{T_A^r}) \simeq \N(\frT_1) \simeq \N(\frT_2) \simeq  \N(K_m^{L_r(A)})$.

 			\end{proof}

 \begin{proof}[Proof of Theorem \ref{cormain}]

From Theorem \ref{generalmain}, $\N(K_m^{T_A^r}) \simeq \N(K_m^{L_r(A)})$. Since $A = \{0, \ldots, i\}$, we observe that $K_2 \times L_r(A)$
is folded onto a subgraph isomorphic to $K_2$. Hence, from Propositions \ref{prop:folding} and \ref{prop:shifting}, we conclude that  $\N(K_m^{T_A^r}) \simeq \N(K_m^{L_r(A) } ) \simeq$ Hom$(K_2, K_m^{L_r(A)} )\simeq$ 
	Hom$(K_2 \times L_r(A), K_m) \simeq$ Hom$(K_2, K_m) \simeq   \N(K_m) \simeq \mathbb{S}^{m-2}$. 
 	 \end{proof}

 \begin{proof}[Proof of Corollary \ref{corsharpe}]

 		Let $\chi(T) = n$. We first show that $\chi(K_{m}^{T_A^r}) = m$. Since $T$ is a perfect graph, it has a subgraph isomorphic to $K_n$. Without loss of generality we assume that $K_n$ is a subgraph of $T$. 	From Lemma \ref{foldgeneral}, $K_{m}^{T_A^r}$ is folded onto a subgraph $\frG$, where $f \in V(\frG)$ if and only if, for each $0 \leq l \leq r-1$, the map  $f_l$
 	is either a constant map or  a graph homomorphism from $T$ to $K_{m}$. Choose $0 \leq l \leq r-1$ such that $l \in A$. Define $\phi: V(\frG) \to [m]$ by

 	\begin{center}
 		$\phi(f)= \begin{cases}
 		f_l(1) & \text{if}  \ f_l \ \text{is a constant map},\\
 		1, & \text{otherwise}.\\
 		\end{cases}$
 	\end{center}	
 	
 	Clearly $\phi$ is well defined. 	Let $f , h \in V(\frG)$ such that $f \sim h$. For any $g \in V(\frG)$ such that $g_i$  is non-constant for some $0 \leq i \leq r-1$,  $N_{K_m^{T_A^r}}(g) = 0$  from Lemma \ref{neighborhood}  and therefore $N_{\frG}(g) \subseteq N_{K_{m}^{T_A^r}}(g)$ implies that $N_{\frG}(g) = \emptyset$.  Hence, $f_l$ and $h_l$ must be  constant maps.  Since $n \geq 2$,
 	and $(1, l) \sim (2, l)$ in $T_A^r$, $f((1, l)) \neq h((2, l)) = h((1,l))$. Hence $\phi(f) \neq \phi(h)$.  It follows that 
 	$\phi$ is a graph homomorphism. 
 	
 	Thus, $\chi(\frG) \leq m$. Since,  $\frG$ contains a subgraph isomorphic to $K_{m}$, namely induced by all constant maps, we conclude that 
 	$\chi(\frG) = m$. Thus, $\chi(K_{m}^{T_A^r}) = \chi(\frG) = m$.

 Since, $conn(\mathbb{S}^{m-2}) = m-3$, result follows from Theorem \ref{cormain}.

 	\end{proof}

\begin{rmk} 
		In the proof of Corollary \ref{corsharpe}, we have shown that  $\chi(K_m^{T_A^r}) = m$ for  $m \leq \chi(T)$. But, this is directly follows from Theorem \ref{thm:moroli}. For, first observe that $T_A^r$ contains $M_r(K_n)$ as a  subgraph, where $n = \chi(T)$. Since, $M_r(K_n)$ satisfies the condition of Theorem \ref{thm:moroli}, we conclude that $\chi(K_m^{M_r(K_n)}) = m$ for all $m \leq n$. Now, since $i^{\ast} : K_m^{T_A^r} \to K_m^{M_r(K_n)}$ defined by $i^{\ast}(f)(x) = f(x)$ is a graph homomorphism, it follows that $\chi(K_m^{T_A^r}) = m$.

\end{rmk}

We now fix some notations.  Let $m$ and $n$ be positive integers such that $m  \leq n+1$. For any $f \in V(K_m^{M(M(K_n))})$ and $i, j \in \{0,1\}$, we  let  $f_{i,j}: [n] \to [m]$  defined by $f_{i,j}(x) = f(x,i,j)$ for all $x \in [n]$. 
For $i \in \{0,1\}$, we let  $w_i= (\ast, i) \in V(M(M(K_n)))$, where $\ast \in V(M(K_n))$ obtained by identifying all the vertices whose second coordinate is $2$. Let $w_2 \in V(M(M(K_n)))$  obtained by identifying all the vertices whose second coordinate is $2$, i.e., by identifying all  the vertices of the form $(x,2)$ where $x \in V(M(K_n))$.

%We first prove few lemmas, which we need in the proof of Theorem \ref{doublenew}.

\begin{lem} \label{folddouble}
	Let $f \in V(K_m^{M(M(K_n))})$. Let  there exist $p, q \in \{0,1\}$ and $i_0, i_1 \in [n], i_0 \neq i_1$
	such that $f_{p,q}(i_0) = f_{p,q}(i_1)$. Then there exists $\tilde{f} \in V(K_m^{M(M(K_n))})$  such that $\tilde{f}_{p,q}(j) = f_{p,q}(i_0)$ for all $j \in [n]$ and $N_{K_m^{M(M(K_n))}}(f) \subseteq N_{K_m^{M(M(K_n))}}(\tilde{f})$.
\end{lem}

\begin{proof}
	
	Define 	$\tilde{f} : V(M(M(K_n))) \to [m]$  by 
	\begin{center}
		$\tilde{f}(x)= \begin{cases}
		f_{p,q}(i_0), & \text{if}  \ x = (j,p,q) \ \text{for some} \ j \in [n],\\
		f(x), & \text{otherwise}. \\
		\end{cases}$
	\end{center}
	
	Let $h \in N_{K_m^{M(M(K_n))}}(f)$.   We show that $h \sim \tilde{f}$.  Let $x = (j_0, p, q)$ for some $j \in [n]$ and let $y \sim x$. Observe that, to prove  $h \sim \tilde{f}$, it is enough to show that $h(x) \neq \tilde{f}(y)$ and $h(y) \neq \tilde{f}(x) = f_{p,q}(i_0) $. If $y = (j, p,q)$ for some $j \in [n]$, then $\tilde{f}(y) = f_{p,q}(i_0) $. Since, $y \sim x$, we conclude that $x$ is adjacent to either $(i_0,p,q)$ or $(i_1,p,q)$. Hence,  $h \sim f$ and $f((i_0, p,q)) = f((i_1, p,q))$ implies that  $h(x) \neq f((i_0,p,q)) = f_{p,q}(i_0) = \tilde{f}(y)$.  If $y \neq (j,p,q)$ for all $j \in [n]$, then $\tilde{f}(y) = f(y)$ and therefore $h(x) \neq \tilde{f}(y)$. Thus, $h(x) \neq \tilde{f}(y)$.
	
	 If $y = w_i$  for some $i \in \{0,1,2\}$, the since $y \sim x$, we see that $y \sim (i_0, p,q)$. Hence, $f \sim h$ implies that $h(y) \neq f((i_0, p,q)) = f_{p,q}(i_0)$. So, assume that $y \neq w_i$ for all $i \in \{0,1,2\}$. For any $i \neq j$, since $i \sim j $ in $K_n$, we observe that either $y \sim (i_0, p, q)$ or $y \sim (i_1, p,q)$. Since, $f \sim h$ and  $f((i_0,p,q))= f((i_1,p,q))$, we conclude that $h(y) \neq f_{p,q}(i_0)$. Thus, $h(y) \neq \tilde{f}(x)$.
\end{proof}

\begin{lem} \label{nbdmlessn}
	Let $m \leq n$ and let  $f \in V(K_m^{M(M(K_n))})$. If  there exist $p, q \in \{0,1\}$ such that  $f_{p,q}$  is a  graph homomorphism from $K_n$ to $K_m$, then $N_{K_m^{M(M(K_n))}}(f) = \emptyset$.
\end{lem}
\begin{proof}	Since $f_{p,q}$ is a graph homomorphism from $K_n$ to $K_m$ and $m \leq n$, we conclude that $m  = n$ and  $\text{Im} \  f_{p,q} = [m]$.  Suppose $N_{K_m^{M(M(K_n))}} (f) \neq \emptyset$ and let  $h \in N_{K_m^{M(M(K_n))}}(f)$. Let us first assume that $q = 1$. Then,  $w_{2} \sim (j,p,q)$ for all $j \in [n]$. Hence, $f \sim h$ implies that $h(w_{2}) \notin \text{Im} \ f_{p,q} = [m]$, which is a contradiction.
	
	Now, let $q = 0$. If $p = 1$, then since $w_0 \sim (j,1,0)$ for all $j \in [n]$, we conclude that $h(w_0) \notin \text{Im} \ f_{p,q} = [m]$, a contradiction. So, assume  $p = 0$.  For any $i \neq j$, since  $(i,1,0) \sim (j,0,0)$ and $f \sim h$, we have $h_{1,0}(i) \neq f_{0,0}(j)$. Hence, $h_{1,0}(i) = f_{0,0}(i)$ for all $i$. Now, since $w_0 \sim (i, 1, 0)$ for all $i$, we see that $f(w_0) \notin$ \text{Im} $h_{1,0} = \text{Im} \ f_{0,0} = [m]$, which is a contradiction.
\end{proof}

\begin{proof}[Proof of Theorem \ref{doubesharp}]
	Using Lemma \ref{folddouble}, $K_m^{M(M(K_n))}$ is folded onto an induced subgraph  $\frG_1$, where $f \in V(\frG_1)$ if and only if, for each $p, q \in \{0,1\}$ the map $f_{p,q}$  is either a constant map or a graph homomorphism from $ K_n$  to $K_m$. Let $\frG_2$ be the induced subgraph of $\frG_1$, where $f \in V(\frG_2)$ if and only if $f_{p,q}$ is a constant map for all $p, q \in \{0,1\}$.   From  Lemma \ref{nbdmlessn}, $N_{\frG_1}(f) = \emptyset$ for all $f$ such that $f_{p,q}$ is non-constant for some $p,q \in \{0,1\}$. Hence, $\N(\frG_1) \simeq \N(\frG_2)$.  It is  easy to see that   $\frG_2 \cong K_m^{M(L_2(\{0\}))}$. Observe that $M(L_2(\{0\}))$ is folded onto $L_1(\{0\})$ and therefore $K_2 \times M(L_2(\{0\}))$ is folded onto $K_2 \times L_1(\{0\})$. Since, $K_2 \times L_1(\{0\})$ is folded onto  a subgraph isomorphic to $K_2$, we have, 
	
	$\N(K_m^{M(M(K_n))}) \simeq  \N(\frG_1) \simeq \N(\frG_2) \simeq \N(K_m^{M(L_2(\{0\}))})  \simeq $ Hom$(K_2, K_m^{M(L_2(\{0\}))}) \simeq$ Hom $(K_2 \times M(L_2(\{0\})), K_m) \simeq $ Hom$(K_2 \times L_1(\{0\}), K_m) \simeq $ Hom$(K_2, K_m) \simeq \N(K_m) \simeq \mathbb{S}^{m-2}$. 
	
	Result follows from  Theorem \ref{doublenew}.
	\end{proof}

\subsection{Proofs of the results of Section \ref{subsec:testgraph}}

As in the previous Section, let $m$ and $n$ be positive integers such that $m \leq n+1$. For any  $f \in V(K_m^{M(M(K_n))})$  and $i, j \in \{0,1\}$, we define $f_{i,j}: [n] \to [m]$ by $f_{i,j}(x) = f(x,i,j)$ for all $x \in [n]$.   For $i \in \{0,1\}$,  $w_i= (\ast, i) \in V(M(M(K_n)))$, where $\ast \in V(M(K_n))$ obtained by identifying all the vertices whose second coordinate is $2$ and  $w_2 \in V(M(M(K_n)))$ is obtained by identifying all the vertices  of the form $(x,2)$, where $x \in V(M(K_n))$.

%\begin{lem} \label{folddouble}
%	Let $f \in V(K_m^{M(M(K_n))})$. Let  there exists $l \in \{0,1\}$ and $i_0, i_1 \in [n], i_0 \neq i_1$
%such that $f_{l,0}(i_0) = f_{l,0}(i_1)$. Then there exists $\tilde{f} \in V(K_m^{M(M(K_n))})$  such that $\tilde{f}_{l,0}(j) = f_{l,0}(i_0)$ for all $j \in [n]$ and $N_{K_m^{M(M(K_n))}}(f) \subseteq N_{K_m^{M(M(K_n))}}(\tilde{f})$.
%	\end{lem}
%
%\begin{proof}
%	
%		Define 	$\tilde{f} : V(M(M(K_n))) \to [m]$  by 
%	\begin{center}
%		$\tilde{f}(x)= \begin{cases}
%		f_{l,0}(i_0), & \text{if}  \ x = (j,l,0) \ \text{for some} \ j \in [n],\\
%		f(x), & \text{otherwise}. \\
%		\end{cases}$
%	\end{center}
%
%Let $h \in N(f)$.   We show that $h \sim \tilde{f}$.  Let $x = (j_0, l, 0)$ for some $j \in [n]$ and let $y \sim x$. Observe that, to prove  $h \sim \tilde{f}$, it is enough to show that $h(y) \neq f_{l,0}(i_0)$. Since $i \sim j $ in $K_n$ for all $i \neq j$, we observe that either $y \sim (i_0, l, 0)$ or $y \sim (i_1, l,0)$. Since, $f \sim h$ and  $f((i_0,l,0))= f((i_1,l,0))$, we conclude that $h(y) \neq f_{l,0}(i_0)$.
%	\end{proof}

   Let $\frG$ be the induced subgraph of $K_m^{M(M(K_n))}$ where $f \in V(\frG)$ if and only if, for any $i,j \in \{0,1\}$ the map $f_{i,j}$  is either a constant map or a graph homomorphism from $ K_n$  to $K_m$. 

We first prove few lemmas, which we need in the proof of Theorem \ref{doublenew}.

\begin{lem}  \label{generalfold}	Let $f \in V(\frG)$ such that $f_{0,0}$ and $f_{1,0}$ are graph homomorphisms. 
	\begin{itemize}
		\item[(i)] If $f_{0,0}(i_0) = f_{1,0}(j_0)$ for some $i_0 , j_0\in [n], i_0 \neq j_0$, then  there exists $\tilde{f} \in V(\frG)$ such that Im $\tilde{f}_{1,0} = \{f_{0,0}(i_0)\}$ and  $N_{\frG}(f) \subseteq N_{\frG}(\tilde{f})$.
		\item[(ii)] 	If Im $f_{0,0}  = Im \ f_{1,0}$ and $f_{0,0}(i) \neq f_{1,0}(i)$ for some $i \in [n]$, then there exists $\tilde{f} \in V(\frG)$ such that $Im \ \tilde{f}_{1,0} = \{f_{0,0}(i)\}$ and $N_{\frG}(f) \subseteq {N_{\frG}(\tilde{f})}$.
		\item[(iii)] If Im $f_{0,0} \neq Im \ f_{1,0}$ and  $f_{0,0}(i) \neq f_{1,0}(i)$, $f_{0,0}(j) \neq f_{1,0} (j)$ for some $i, j \in [n], i \neq j$, then there exists $\tilde{f} \in V(\frG)$ such that $\tilde{f}_{1,0}$ is a constant map  and $N_{\frG}(f) \subseteq N_{\frG}(\tilde{f})$. 
		\end{itemize}
	\end{lem}

\begin{proof}  %Recall that $w_i= (1,2,i) = (2,2,i) = \ldots = (n,2,i) \in V(M(M(K_n)))$ for $i \in \{0,1,2\}$.
	\begin{itemize}
		\item[(i)] 	Define 	$\tilde{f} : V(M(M(K_n))) \to [m]$  by 
		\begin{center}
			$\tilde{f}(x)= \begin{cases}
			f_{0,0}(i_0), & \text{if}  \ x = (j,1,0) \ \text{for some} \ j \in [n],\\
			f(x), & \text{otherwise}. \\
			\end{cases}$
		\end{center}
		Let $h \sim f$.  Let $x = (j, 1, 0)$ for some $j \in [n]$ and let $y \sim x$. Let $y = (y_1, p,q)$. Observe that  $p \in \{0,2\}$ and $q \in \{0,1\}$.  If $p = 2$, then $y = w_q$. In this case,  $y \sim (j_0,1,0)$ and therefore $f \sim h$ implies that 
		$h(y) \neq f((j_0,1,0)) = f_{1,0}(j_0) = f_{0,0}(i_0)$. Assume that $p = 0$. In this case, since $i_0 \neq j_0$, we see that  either $(y_1, 0, q) \sim (i_0,0,0)$ (if $y_1 \neq i_0$) or $(y_1, 0,q) \sim (j_0, 1,0)$ (if $y_1 \neq j_0$). Since, $f(i_0,0,0) = f(j_0,1,0) $, we derive  that $h(y) \neq f_{0,0}(i_0)$.  Hence, $h(y) \neq \tilde{f}(x)$.
		
		Since, $y \sim x$, $y \neq (i,1, 0)$ for any $i \in [n]$ and therefore $\tilde{f}(y) = f(y)$. Hence, $h(x) \neq \tilde{f}(y)$.  Now, from the definition of $\tilde{f}$, we conclude that $\tilde{f} \sim h$. Thus, $N_{\frG}(f) \subseteq N_{\frG} (\tilde{f})$.
		
		\item[(ii)] Since $\text{Im} \ f_{0,0} = \text{Im} \ f_{1,0}$ and $f_{0,0}(i) \neq f_{1,0}(i)$, we see that there exists $j \neq i$ such that $f_{1,0}(j) = f_{0,0}(i)$. Proof follows from $(i)$.
		
		\item[(iii)] 	Since $f_{0,0}$ is a graph homomorphism from $K_n$ to $K_m$, we see that $f_{0,0}(i) \neq f_{0,0}(j)$. Further, since $f_{1,0}$ is a graph homomorphism, $| \text{Im} \  f| = n$. Now, $m \leq n+1$ implies that $\{ f_{0,0}(i), f_{0,0}(j) \} \cap  \text{Im} \ f_{1,0} \neq \emptyset$.
		
		If $f_{0,0}(i) \in \text{Im} \ f_{1,0}$, then $f_{1,0}(p) = f_{0,0}(i)$ for some $p$. Since, $f_{0,0}(i) \neq f_{1,0}(i)$, we see that $p \neq i$. Similarly, if $f_{0,0}(j) \in \text{Im} \ f_{1,0} $ and $f_{0,0}(j) = f_{1,0}(q)$, then $q \neq j$. Proof follows from $(i)$. 
		\end{itemize}
	\end{proof}

%\begin{lem} \label{nbdmlessn}
%	 Let $m \leq n$.  If $f \in V(\frG)$ such that $f_{0,0}$  is a  graph homomorphism, then $N_{\frG}(f) = \emptyset$.
%	\end{lem}
%\begin{proof}	Since $f_{0,0}$ is a graph homomorphism, $m  = n$. Suppose $N_{\frG} (f) \neq \emptyset$ and let  $h \in N_{\frG}(f)$. For any $i \neq j$, since  $(i,1,0) \sim (j,0,0)$ and $f \sim h$, we have $h_{1,0}(j) \neq f_{0,0}(i)$. Hence, $h_{1,0}(i) = f_{0,0}(i)$ for all $i$. Now, since $w_0 \sim (i, 1, 0)$ for all $i$, we see that $f(w_0) \notin$ \text{Im} $h_{1,0} = \text{Im} \ f_{0,0} = [m]$, which is a contradiction.
%	\end{proof}
\begin{lem} \label{neighborhodgeneral}
	Let $f \in V(\frG)$ such that $f_{0,0}$  is a  graph homomorphism.
	\begin{itemize}
	
		\item[(i)] Let $f_{1,0}$ is a graph homomorphism such that  $\text{Im} \ f_{1,0} = \text{Im} \ f_{0,0}$. If  $f(w_0) \notin Im \ f_{0,0}$, then $N_{\frG}(f) = \emptyset$.
		\item[(ii)] Let $f_{1,0}$ is a graph homomorphism, $\text{Im} \ f_{0,0} \neq$  Im $f_{1,0}$ and $f(w_0) \in \text{Im} \ f_{0,0}$. If $f_{0,0}$ and $f_{1,0}$ differ at only one vertex, then  $N_{\frG}(f) = \emptyset$. 
		\item[(iii)]  Let $f_{1,0}$ is a constant map such that  $ f_{1,0}(1) \notin$ \text{Im} $f_{0,0}$.  If $f(w_0) \in $ \text{Im} $f_{0,0}$, then $N_{\frG}(f) = \emptyset$.
		\end{itemize}
	\end{lem}

\begin{proof}
	If $m \leq n$, then since $f_{0,0}$ is  a graph homomorphism, $N_{\frG}(f) = \emptyset$ from Lemma \ref{nbdmlessn}. So, assume that $m = n+1$.  Suppose $N_{\frG}(f) \neq \emptyset$ and let $h \in N_{\frG}(f)$.  Observe that, since $w_2 \sim (i,j,1)$ for all $1 \leq i, j \leq n$ and $h \sim f$, $f(w_2) \notin$ Im $h_{0,1} \cup$ Im $ h_{1,1} \cup h(w_1)$. Therefore,  Im $h_{0,1} \cup$ Im $h_{1,1} \cup h(w_1) \neq [m]$. 
	\begin{itemize}
	
		\item[(i)] Since $f(w_0) \notin \text{Im} \ f_{0,0}$, $|\text{Im} \ f_{0,0}| = n$ and $m = n+1$, we see that $\{f(w_0)\} = [m] \setminus \text{Im} \ f_{0,0}$.  Since, $(i,1,1) \sim (j,0,0)$ for $i \neq j$ and $h \sim f$, we see that $h((i,1,1)) \in \{f_{0,0}(i), f(w_0)\}$. For any $i \in [n]$, since $(i,1,1) \sim w_0$ , we conclude that $f(w_0) \neq h((i,1,1))$.  Hence, $h((i,1,1)) = f((i,0,0))$ for all $i$, i.e., $h_{1,1} = f_{0,0}$. Since $w_1 \sim (i, 1,0)$ for all  $i \in [n]$, $h \sim f$  implies that $h(w_1) \notin \text{Im} \ f_{1,0} = \text{Im} \ f_{0,0} = \text{Im} \ h_{1,1}$. Thus, $\text{Im} \ h_{1,1} \cup h(w_1) = [m]$, which implies that Im $h_{0,1} \cup$ Im $h_{1,1} \cup h(w_1) = [m]$, a contradiction. Therefore, $N_{\frG}(f) = \emptyset$.

		\item[(ii)]  
		Let $\{a\} = [m] \setminus \text{Im} \ f_{0,0}$. Since $\text{Im} \ f_{0,0} \neq \text{Im}  \ f_{1,0}$ and $|\text{Im} \  f_{1,0}| = n$,  there exists $i_0 \in [n]$ such that $f_{1,0}(i_0) = a$.  Further, since $f_{0,0}$ and $f_{1,0}$ differs at only one vertex, we conclude that $f_{0,0}(j) = f_{1,0}(j)$ for all $j \in [n], j \neq i_0$. 
		
		Since $(i,0,1) \sim (j,0,0)$ for all $i \neq j$ and $h \sim f$, we see that $h((i,0,1) )\in \{f_{0,0}(i), a\}$. If $i \neq i_0$, then since $(i,0,1) \sim (i_0, 1,0)$ and $f((i_0,1,0)) = f_{1,0}(i_0) = a$, we see that $h((i,0,1)) \neq f((i_0,1,0)) = a$. Hence, $h((i,0,1)) = f_{0,0}(i)$ for all $i \neq i_0$ and $h((i_0,0,1)) \in \{f_{0,0}(i_0), a\}$. Thus, $h_{0,1}$ is either $f_{0,0}$ or $f_{1,0}$.
		
		Let $h_{0,1} = f_{1,0}$. Since $w_1 \sim (i, 1,0)$ for all $i$, we see that $h(w_1) \notin \text{Im} \ f_{1,0}$ and therefore $\text{Im} \ h_{0,1} \cup h(w_1) = [m]$, which is a contradiction.  
		
		Now, let $h_{0,1} = f_{0,0}$. If $ h_{1,1} $ is constant map, then since $(i,1,1) \sim (j,0,0)$ for all $i \neq j$, $h \sim f$ implies that Im $  h_{1,1} = a$. Hence $\text{Im} \ h_{0,1} \cup \text{Im}  \ h_{1,1} = [m]$, a contradiction. So assume that  $h_{1,1}$ is non-constant. Since, $h \in V(\frG)$, $h_{1,1}$ is a graph homomorphism. For any $i \in [n]$, $w_0 \sim (i,1,1)$ and therefore  $f(w_0) \notin$ Im $ h_{1,1}$. Since, $f(w_0) \in$ Im $f_{0,0}$, we see that $f(w_0) \neq a$. Now, $ |$Im $h_{1,1} | = n$ implies that $a \in$ Im $ h_{1,1}$ and therefore Im $h_{0,1} \cup$  Im $h_{1,1} = [m]$, which is not possible. Hence, $N_{\frG}(f) = \emptyset$.
		
		\item[(iii)] Let $\{a\} = $ Im $f_{1,0}$. Then $a \notin $ Im $f_{0,0}$.   Since $(i,0,1) \sim (j,1,0)$ for all $i \neq j$, $h \sim f$ implies that $a \notin $ Im $h_{0,1}$. Further, since $(i,0,1) \sim (j,0,0)$ for all $i \neq j$, we conclude that $h_{0,1}(k) = f_{0,0}(k)$ for all $k \in [n]$. Hence, $h_{0,1} = f_{0,0}$.
		
		Let us first assume that  $h_{1,1}$ is a constant map and Im $h_{1,1} = \{b\}$. Observe that   $b \notin$ Im $f_{0,0}$. Thus,  $h_{0,1} = f_{0,0}$ implies that Im $h_{0,1} \cup \{b\} = [m]$, which is a contradiction. 
		
		Now, assume that $h_{1,1}$ is a graph homomorphism. Observe that $f(w_0) \notin $ Im $h_{1,1}$.  Since $f(w_0) \in \text{Im} \ f_{0,0}$, there exists $i_0$  such that $f_{0,0}(i_0) = f(w_0)$. Since, $(i_0, 1,1) \sim (j, 0,0)$, $h_{1,1}(i_0) \neq f_{0,0}(j)$ for all $j \neq i_0$. Hence, either $h_{1,1}(i_0) = f_{0,0}(i_0)$ or $h_{1,1}(i_0) = a$. But, since, $f_{0,0}(i_0) = f(w_0) \notin \text{Im} \ h_{1,1}$, we conclude that $h_{1,1}(i_0) = a$ i.e., $a \in$ Im $h_{1,1}$. Thus, Im $h_{0,1} \cup $ Im $h_{1,1} = \text{Im} \ f_{0,0} \cup \text{Im} \ h_{1,1} = [m]$, which is not possible. Hence, $N_{\frG}(f) = \emptyset$.

		\end{itemize}
	\end{proof}

\begin{proof}[Proof of Theorem \ref{doublenew}]
We first prove that, for any $m \leq n+1$, $\chi(K_m^{M(M(K_n))}) = m$.	From Lemma \ref{folddouble}, $K_m^{M(M(K_n))}$ is folded onto a subgraph $\frG$, where $f \in V(\frG) $ if and only if for any $i,j \in \{0,1\}$ the map $f_{i,j}$  is either a constant map or a graph homomorphism from $ K_n$  to $K_m$. 
	Using Lemma \ref{generalfold}, $\frG$ is folded onto the induced subgraph on vertex set $V(\frG_1)$, where  $f \in V(\frG_1)$ if and only if the following are true:
\begin{itemize}
	\item[(A)] if  Im $f_{0,0} = \text{Im} \ f_{1,0}$, then $f_{0,0} = f_{1,0}$ (from Lemma \ref{generalfold} $(ii)$).
	\item[(B)]  if $f_{0,0}, f_{1,0}$ are graph homomorphisms and Im $f_{0,0} \neq$ Im $f_{1,0}$, then $f_{0,0}$
	and $f_{1,0}$ differ at only one vertex (from Lemma \ref{generalfold} $(iii)$).
	\end{itemize}	

We define the subsets $U_1, U_2, U_3 \subseteq V(\frG_1)$ by 

$U_1 = \{f \in V(\frG_1) \ | \ f_{0,0}, f_{1,0} \ \text{are graph homomorphisms}, \text{Im} \ f_{0,0} = \text{Im}  \ f_{1,0} \ \text{and} \  f(w_0) \notin \text{Im} \ f_{0,0}\}$,

$U_2 = \{f \in V(\frG_1) \ | \ f_{0,0}, f_{1,0} \ \text{are graph homomorphisms}, \text{Im}\ f_{0,0} \neq \text{Im}  \ f_{1,0}, f(w_0) \in\text{Im}\ f_{0,0} \ \text{and}, f_{0,0}\ \text{and} \ f_{1,0} \ \text{differ at only one vertex}\}$ 
and

$U_3 = \{f \in V(\frG_1) \ | \ f_{0,0} \ \text{is a graph homomorphism}, \ f_{1,0} \ \text{is a constanta map}, \ f_{1,0}(1) \notin \text{Im} \ f_{0,0}\ \text{and} \ f(w_0) \in \ \text{Im} \  f_{0,0} \}$.

Let $\frG_2$ be the induced subgraph of $\frG_1$ on vertex set $V(\frG_1) \setminus (U_1 \cup U_2 \cup U_3)$. From Lemma \ref{neighborhodgeneral},   $N_{\frG}(f) = \emptyset$ for all $f \in U_1 \cup U_2 \cup U_3$. Since $N_{\frG_1}(f) \subseteq N_{\frG}(f)$, we  see that $N_{\frG_1}(f) = \emptyset$  for all $f \in U_1 \cup U_2 \cup U_3$. Thus, we conclude  that $\chi(\frG_1) = \chi(\frG_2)$.

We define $\phi: V(\frG_2) \to [m]$ by 
	\begin{center}
	$\phi(f)= \begin{cases}
	f_{0,0}(1), & \text{if}  \ f_{0,0} \ \text{is a constant map},\\
	f_{1,0}(1), & \text{if} \ f_{0,0} \ \text{is a graph homomorphism and} \ f_{1,0}\  \text{is a constant map},\\
	f(w_0), & \text{if} \ f_{0,0} \ \text{and} \ f_{1,0} \ \text{are graph homomorphisms}.\\
	\end{cases}$
\end{center}

Clearly, $\phi$ is well defined. We show that $\phi$ is a graph homomorphism from $\frG_2$ to $K_m$.
Let $(f, h) \in E(\frG_2)$. Let us first assume that $m \leq n$. In this case, from Lemma \ref{nbdmlessn}, $N_{\frG_2}(g) = \emptyset$ for all $g$ such that $g_{0,0}$ is a graph homomorphism. Hence, $f_{0,0}$ and $h_{0,0}$ must be constant maps. Here, $\phi(f) = f_{0,0}(1)$  and $\phi(h) = h_{0,0}(1)$.  Since, $(1,0,0) \sim (2,0,0)$, $f_{0,0}(1) = f((1,0,0)) \neq h((2,0,0)) = h((1,0,0)) = h_{0,0}(1)$. Hence $\phi(f) \neq \phi(h)$. 

So, assume that $m = n+1$. We consider the following cases.

{\it Case 1.}  $f_{0,0}$ is a constant map.
	 
In this case $\phi(f) = f_{0,0}(1)$. If $h_{0,0}$ is a constant map, then $\phi(h) = h_{0,0} (1)$. Since $(1,0,0) \sim (2,0,0)$, $f((1,0,0)) \neq h((2,0,0)) = h((1,0,0))$. Hence $\phi(f) \neq \phi(h)$. If $h_{0,0}$ is  a graph homomorphism  and $h_{1,0}$ is  a constant  map, then $\phi(h) = h_{1,0}(1)$. In this case, since $(1,0,0) \sim (2,1,0)$, we see that $f((1,0,0)) \neq h((2,1,0)) = h((1,1,0))$ and therefore $\phi(f) \neq \phi(h)$. Now, assume that both $h_{0,0}$ and $h_{1,0}$ are graph homomorphisms. 
Here, $\phi(h) = h(w_0)$.

{\it Subcase 1.1.}	Im $h_{0,0} =$ Im  $h_{1,0}$.

In this case from the condition $(A)$ given above, we conclude that $h_{0,0} = h_{1,0}$. Since, $h \notin U_1$, we see that $h(w_0) \in $ Im $h_{0,0}$. Since $f_{0,0}$ is a constant map and $h \sim f$, observe that 
$f_{0,0}(1) \notin $ Im $h_{0,0}$. Thus $\phi(f) = f_{0,0}(1) \neq h(w_0) = \phi(h)$.

{\it Subcase 1.2.} 	Im $h_{0,0} \neq $ Im $h_{1,0}$.

Since, $h_{0,0}$ and $h_{1,0}$ are graph homomorphisms, $m = n+1$ and  $\text{Im} \ h_{0,0} \neq \text{Im} \ h_{1,0}$, we see that $\text{Im} \ h_{0,0} \cup \text{Im} \ h_{1,0} = [m]$. From the condition $(B)$ given above, we conclude that $h_{0,0}$ and $h_{1,0}$ differ at only one vertex. Since, $h \notin U_2$, we see that $h(w_0) \notin$ Im $h_{0,0}$. Now, $\text{Im} \ h_{0,0} \cup \text{Im} \ h_{1,0} = [m]$ implies that $h(w_0) \in$ Im $h_{1,0}$. Since, $f_{0,0}$ is a constant map  and $(i,0,0) \sim (j,1,0)$ for all $i \neq j$, we see that
$f \sim h \implies f_{0,0}(1) \notin $ Im $h_{1,0}$. Hence, $\phi(f) \neq \phi(h)$.

	 {\it Case 2.} $f_{0,0}$ is a graph homomorphism and $f_{1,0}$ is a constant map.

In this case  $\phi(f)  = f_{1,0}(1)$.  If $h_{0,0}$ is a constant map then $\phi(h) = h_{0,0}(1)$.  Since $(i,0,0) \sim (j,1,0)$ for all $i \neq j$ and $f \sim h$, we conclude that $h_{0,0}(1) \neq f_{1,0}(1)$.

 Let $h_{0,0}$ is  a graph homomorphism and $h_{1,0}$ is a constant map. Here, $\phi(h) = h_{1,0}(1)$.  Since $(i,0,0) \sim (j,1,0)$ for all $i \neq j$ and $f \sim h$, we see that  $h_{1,0}(1) \notin $ Im $f_{0,0}$. If $f_{1,0}(1) \in$ Im $f_{0,0}$, then $\phi(h) \neq \phi(f)$. So assume that $f_{1,0}(1) \notin $ Im $f_{0,0}$.
Since $f \notin U_3$, $f(w_0) \notin $ Im $f_{0,0}$. Using the fact that $|\text{Im} \ f_{0,0}| = n$ and $m = n+1$, we conclude that $f(w_0) = f_{1,0}(1) = \phi(f)$. Since $w_0 \sim (i,1,0)$ for all $i$, we see that $h_{1,0}(1) \neq f(w_0)$. Thus $\phi(f) \neq \phi(h)$.

Now, let both $h_{0,0}$ and $h_{1,0}$ are graph homomorphisms. Here, $\phi(h) = h(w_0)$. Since $w_0 \sim (i,1,0)$ for all $i$, we have  $f_{1,0}(1) \neq h(w_0)$ and therefore  $\phi(f) \neq \phi(h)$.

{\it Case 3.}	Both $f_{0,0}$ and $f_{1,0}$ are graph homomorphisms.

In this case $\phi(f) = f(w_0)$. If $h_{0,0}$ is a constant map or $h_{1,0}$ is a constant map, then by reversing the roles of $h$ and $f$, we conclude from Case $1$ and Case $2$ that $\phi(f) \neq \phi(h)$. So assume that both $h_{0,0}$ and $h_{1,0}$ are graph homomorphisms. Here, $\phi(h) = h(w_0)$.

{\it Subcase 3.1.} Im $f_{0,0} = $ Im $f_{1,0}$.

From $(A)$, $f_{0,0} = f_{1,0}$.  Since $f \notin U_1$, $f(w_0) \in$ Im $f_{0,0}$. Further, since $w_0 \sim (i,1,0)$ for all $i$, we see that $h(w_0) \notin \text{Im}\ f_{1,0}$. Since $f_{0,0} = f_{1,0}$, $f(w_0) \in $ Im $f_{1,0}$. Hence $\phi(f) = f(w_0) \neq h(w_0) = \phi(h)$.

{\it Subcase 3.2.}
	Im $f_{0,0} \neq $ Im $f_{1,0}$.

From $(B)$, $f_{0,0}$ and $f_{1,0}$ differ at only one vertex.  Since, $f \notin U_2$, $f(w_0) \notin$ Im  $f_{0,0}$. Since Im $f_{0,0} \neq $ Im $f_{1,0}$, we see that $f(w_0) \in$ Im $f_{1,0}$. Observe that $h(w_0) \notin $ Im $f_{1,0}$. Hence $\phi(f) \neq \phi(h)$. 

Thus $\phi(f) \sim \phi(h)$ in $K_m$. Therefore $\phi$ is a graph homomorphism. Hence, $\chi(\frG_2) \leq m$. Since $\frG_2$ contains a clique of order $m$, namely induced by constant maps, we conclude that
$\chi(\frG_2) = m$. Since folding preserve chromatic number, we have $\chi(K_m^{M(M(K_n))}) = \chi(\frG) = \chi(\frG_1) = \chi(\frG_2) = m$.

Now, result follows from the fact  that  we have graph homomorphism $i^{\ast} : K_m^G \to K_m^{M(M(K_n))}$ defined by $\phi(f)(x)= f(x)$ and $K_m^G$ has a subgraph isomorphic to the complete graph $K_m$, namely induced by all constant maps.
	\end{proof}

In \cite{TM}, Matsushita proved the following sufficient condition for a graph to be a homotopy test graph.

\begin{thm} \cite[Theorem $7$]{TM} \label{retract} 	Let $A$ be a graph of chromatic number  $n$. If $A$ has a subgraph isomorphic to $K_n$, then $A$ is a homotopy test graph.
\end{thm}

\begin{proof}[Proof of Corollary \ref{testgraph}] Since, $K_m^G$ has a subgraph isomorphic to $K_m,$ result follows from  Theorem \ref{doublenew} and Theorem  \ref{retract}.
	
\end{proof}

\subsection{Proof of the result of Section \ref{hedetnimi}}

\begin{proof}[Proof of Theorem \ref{Hedet}]
	
	In \cite{ZS}, El-Zahar and Sauer (see also Proposition $2.2$, \cite{CT} showed that the Conjecture \ref{Hedetniemi} is equivalent to the statement that $\chi(K_m^A) = m$ for all $m < \chi(A)$. 
		They first proved and  then used the fact that for any graph homomorphism $\phi: A \times B \to K_m$, there are induced graph homomorphisms $\phi_A : A \to K_m^B$ defined by $\phi_A(a) (b)= \phi(a,b)$ and $\phi_B: B \to K_m^A$ defined by $\phi_B(b)(a) = \phi(a,b)$. We use this fact in this proof.
	
%	Let us first assume that $M_r(K_n)$ is a subgraph of $G$ and $\chi(G \times H) = n$.  Let $\phi: G \times H \to K_n$ be a graph homomorphism.We have the induced graph homomorphism $\phi_H: H \to K_n^G$. By  taking $T = K_n$ and $A = \{0\}$, we have $\chi(K_n^G) = n$ from Theorem \ref{universal}. Hence,   $\chi(H) \leq \chi(K_n^G) = n$. Observe that $\chi(M_r(K_n)) = n+1$ and therefore $\chi(G) \geq n+1$. Further, $\chi(G \times H) = n$ implies that $\chi(H) \geq n$. Thus 
%	$\chi(H) = n$, which implies that $\min \{\chi(G), \chi(H)\}  = \chi(H) = n$.
%	
	
Since $\chi(M(M(K_n))) = n+2$, we see that $\chi(G) \geq n+2$.  Let  $H$ be a graph such that $\chi(G \times H) = n+1$. 	Let $\psi: G \times H \to K_{n+1}$ be  a graph homomorphism. From Theorem \ref{doublenew}, $\chi(K_{n+1}^G) = n+1$. Since, we have the  graph homomorphism $\psi_H: H \to K_{n+1}^G$, $\chi(H) \leq n+1$.  Using the fact that  $\chi(G \times H) \leq \chi(H)$, we conclude that  $\min \{\chi(G), \chi(H)\}  = \chi(H) = n+1$.
\end{proof}

\bibliographystyle{plain}

\begin{thebibliography}{99}
\begin{small}



\bibitem{meysam} M. Alishahi and H. Hajiabolhassan, \newblock\emph{Hedetniemi's conjecture via altermatic number}, \newblock{ https://arxiv.org/abs/1403.4404}.




\bibitem{BK}  E.  Babson and D.  N. Kozlov, \newblock\emph{Complexes of graph homomorphisms}. \newblock{Israel J. Math.} Vol 152, 2006, 285--312.


\bibitem{BK1} E. Babson and D. N. Kozlov, \newblock\emph{Proof of the Lova\'sz conjecture}.\newblock{Annals of Math. (2)} Vol 165, 2007, 965--1007.





%\bibitem{BK2} Eric Babson and Dmitry N. Kozlov, \newblock\emph{Topological obstructions to graph colorings}. \newblock{Electron. Res. Announc. Amer. Math. Soc.} Vol 9, 2003, 61-68.
\bibitem{bj} A. Bj{\" o}rner, \newblock\emph{Topological methods}. \newblock{Handbook of combinatorics, Elsevier, Amsterdam} Vol 1,2, 1995, 1819--1872.



\bibitem{BL} A. Bj\"{o}rner and M. de Longueville, \newblock\emph{Neighborhood complexes of stable Kneser graphs}.
\newblock{Combinatorica}, 23(1), 2003, 23--34.



\bibitem{moroli} I. Broere and M. D. V. Matsoha, \newblock\emph{Applications of Haj{\'o}s-type constructions to the Hedetniemi conjecture}. \newblock{Journal of Graph Theory}, 85.2, 2017, 585--594.

\bibitem{BEL} S. A. Burr, P.  Erd\H{o}s and L. Lov{\'a}sz, \newblock\emph{On graphs of Ramsey type}. \newblock{Ars Combinatoria 1.1}, 1976, 167--190.


\bibitem{Strongperfect} M. Chudnovsky, N. Robertson, P.  Seymour and R.  Thomas, \newblock\emph{The strong perfect graph theorem}. \newblock{Ann. of Math}. (2) 164, no.1, 2006,  51--229.


\bibitem{ad} A. Dochtermann, \newblock\emph{Hom complexes and homotopy type in the category of graphs}. \newblock{European Journal of Combinatorics}, Vol 30, 2009, 490--509.


\bibitem{antonandschultz} A. Dochtermann and C. Schultz. \newblock\emph{Topology of Hom complexes and test graphs for bounding chromatic number}, \newblock{Israel Journal of Mathematics}, 187, 2012, 371--417.


\bibitem{DSW} D. Duffus, B. Sands and R. Woodrow, \newblock\emph{On the chromatic number of the product of graphs}.
\newblock{ J. Graph Theory 9}, no.4, 1985, 487--495.


\bibitem{ZS} M. El-Zahar and N. Sauer, \newblock\emph{The chromatic number of the product of two $4$-chromatic graph is $4$},
\newblock{Combinatorica}, 5, 1985, 121-126.


\bibitem{GR} C. Godsil and G. Royle, \newblock\emph{Algebraic Graph Theory}. \newblock{Graduate Texts in Mathematics, Springer Verlag, New York}. Vol 207, 2001.

\bibitem{h} A. Hatcher, \newblock\emph{Algebraic Topology}. Cambridge University Press, 2002.


\bibitem{HD} S. T. Hedetniemi, \newblock\emph{Homomorphisms and graph automata}, \newblock{University of Michigan Technical Report} 03105-44-T, 1966.


\bibitem{HN} P. Hell and J. Ne{\u s}et{\u r}il, \newblock{Graphs and Homomorphisms}. \newblock{Oxford Lecture Series in Mathematics and its Applications, Oxford University Press, Oxford}. Vol 28, 2004.

\bibitem{Hoory} S. Hoory and N. Linial. \newblock\emph{A counterexample to a conjecture of Bj\"{o}rner and Lov{\'a}sz on the $\chi$-coloring complex} \newblock{Journal of Combinatorial Theory, Series B}, 95.2, 2005, 346--349.





%\bibitem{jj} J. Jonsson, \newblock\emph{Simplicial Complexes of graphs}. \newblock{Lecture Notes in Mathematics, Springer Verlag, Berlin}, Vol 1928, 2008.



\bibitem{MK} M. Kneser, \newblock\emph{Aufgabe 360}. \newblock{Jahresbericht der Deutschen Mathematiker-Vereinigung}, 58, 1955.



\bibitem{kozlov} D. N.  Kozlov, \newblock\emph{A simple proof for folds on both sides in complexes of graph homomorphisms}. \newblock{Proceedings of the American Mathematical Society},134.5, 2006,  1265--1270.


\bibitem{Kozlov} D. N.  Kozlov, \newblock\emph{Chromatic numbers, morphism complexes and Steifel Whitney Classes}. \newblock{Geometric Combinatorics}, IAS Park City Math. Ser,. 13,
Amer. Math. Soc, 2007, 249--315.


%\bibitem{dk1} Dmitry Kozlov, \newblock\emph{Cohomology of colorings of cycles}. \newblock{American Journal of Mathematics}. Vol. 130, 2008,  829--857.















\bibitem{dk} D. N.  Kozlov, \newblock\emph{Combinatorial Algebraic Topology}. \newblock{Springer Verlag, Berlin}  1928, 2008.


\bibitem{l} L.  Lov{\'a}sz,  \newblock\emph{Kneser's conjecture, chromatic number and homotopy}. \newblock{J. Combinatorial Theory Ser. A}, Vol 25, 1978, 319--324.



%\bibitem{sm} Saunders Maclane, \newblock\emph{Categories for the Working Mathematician, second edition}. \newblock{Graduate Texts in Mathematics, Springer-Verlag, New York}. Vol 5, 1998.




%\bibitem{moroli} Moroli D. V. Matsoha, \newblock\emph{Moroli David Vusi, Aspects of graph homomorphisms and the Hedetniemi Conjecture}, \newblock{Ph.D. Thesis, University of Pretoria}, 2016.


\bibitem{TM}T. Matsushita, \newblock\emph{Answers to some problems about graph coloring test graphs}. \newblock{European Journal of Combinatorics}, 45, 2015, 59--64.

%\bibitem{JM} J. R. Munkres, \newblock\emph{Topology}. \newblock{Second edition, Prentice Hall, Upper saddle River}, 2000.


\bibitem{Mycielski} J. Mycielski, \newblock\emph{Sur le coloriage des graphes}. \newblock{Colloq. Math. 3}, 1955, 161--162.

\bibitem{NS} N. Nilakantan and S. Shukla, \newblock\emph{Neighborhood complexes of some  exponential graphs}. \newblock {Electron. J. Combin.}, 23, no.2, 2016,  P2.26. 





\bibitem{sauer} N. Sauer, \newblock\emph{Hedetniemi's conjecture -- a survey}, \newblock{Discrete Math. 229}, 2001, 261--292.




\bibitem{SZ} N. Sauer and X.  Zhu, \newblock\emph{
An approach to Hedetniemi's conjecture}. 
\newblock{J. Graph Theory } 16,  no.5, 1992, 423--436. 



\bibitem{schultz} C. Schultz. \newblock\emph{The equivariant topology of stable Kneser graphs}, \newblock{Journal of Combinatorial Theory Series A}, 118.8, 2011,  2291--2318.

\bibitem{Tardif} C. Tardif, \newblock\emph{Fractional chromatic numbers of cones over graphs}, \newblock{ J. Graph Theory} 38,  2001 87--94.

\bibitem{CT}  C. Tardif, \newblock\emph{Hedetniemi's conjecture, 40 years later}, \newblock{Graph Theory Notes NY} 54, 2008, 46--57.




\bibitem{TZ}C. Tardif  and X. Zhu, \newblock\emph{On Hedetniemi's conjecture and the colour template scheme}. \newblock{Discrete mathematics} 253, no.1-3, 2002, 77--85.

\bibitem{DT} D. Turz{\' i}k, \newblock\emph{A note on chromatic number of direct product of graphs}. \newblock{Comment. Math. Univ. Carolin.} 24, no. 3, 1983, 461--463. 


\bibitem{EW}E. Welzl, \newblock\emph{
Symmetric graphs and interpretations}. \newblock{ 
J. Combin. Theory Ser. B}, 37, no.3,  1984 235--244. 

\bibitem{XZ} X. Zhu, \newblock\emph{A survey on Hedetniemi's conjecture}, \newblock{Taiwanese J. Math}, 2, 1998, 1--24.




%\bibitem{claude} Claude Tardif \newblock\emph{Multiplicative graphs and semi-lattice endomorphisms in the category of graphs}. \newblock{ Journal of Combinatorial Theory, Series B} 95.2 (2005): 338--345.

%\bibitem{marcin} Marcin Wrochna, \newblock\emph{Square-free graphs are multiplicative}.  \newblock{Journal of Combinatorial Theory, Series B }122 (2017): 479--507.



\end{small}
\end{thebibliography}

\end{document}